\documentclass{amsart}

\usepackage{amsmath}
\usepackage{amssymb}
\usepackage{color}
\usepackage{mathrsfs}
\usepackage{url}
\usepackage{enumitem}
\usepackage{stmaryrd}
\usepackage[all]{xy}
\usepackage{verbatim}

\newtheorem{Thm}{Theorem}[section]

\newtheorem{Cor}[Thm]{Corollary}
\newtheorem{Eg}[Thm]{Example}
\newtheorem{Def}[Thm]{Definition}

\newtheorem{Lem}[Thm]{Lemma}
\newtheorem{Prop}[Thm]{Proposition}

\newtheorem{Rmk}[Thm]{Remark}

\newcommand{\Cp}{\mathbb{C}_p}
\newcommand{\FF}{\mathbb{F}}

\newcommand{\OCp}{\mathcal{O}_{\Cp}}
\newcommand{\Qp}{\mathbb{Q}_p}
\newcommand{\QQ}{\mathbb{Q}}

\newcommand{\ZZ}{\mathbb{Z}}

\author{Christopher Davis}
\address{University of California, Irvine, Dept of
Mathematics, Irvine, CA 92697}
\email{davis@math.uci.edu}
\date{\today}

\author{Kiran S. Kedlaya}
\address{University of California, San Diego, Dept of Mathematics, La Jolla, CA 92093}
\email{kedlaya@ucsd.edu}

\thanks{Davis was supported by the Max-Planck-Institut f\"ur Mathematik (Bonn).
Kedlaya was supported by NSF (CAREER grant DMS-0545904, grant DMS-1101343), 
DARPA (grant HR0011-09-1-0048),
MIT (NEC Fund, Green professorship), 
UCSD (Warschawski professorship).}

\CompileMatrices

\begin{document}

\title{On the Witt vector Frobenius}

\begin{abstract}
We study the kernel and cokernel of the Frobenius map 
on the $p$-typical Witt vectors of a commutative ring, not necessarily
of characteristic $p$. 
We give some equivalent conditions to surjectivity of the Frobenus map on both finite and infinite length
Witt vectors; the former condition turns out to be stable under certain integral extensions,
a fact which relates closely to a generalization of Faltings's almost purity theorem.
\end{abstract}

\maketitle

\section*{Introduction}

Fix a prime number $p$. To each ring $R$ (always assumed commutative and with unit), we may associate in a functorial
manner the ring of \emph{$p$-typical Witt vectors} over $R$, denoted $W(R)$, and an endomorphism $F$ of $W(R)$
called the \emph{Frobenius endomorphism}. The ring $W(R)$ is set-theoretically
an infinite product of copies of $R$, but with an exotic ring structure; for example, for $R$ a perfect ring of characteristic $p$, $W(R)$ is the unique strict $p$-ring with $W(R)/pW(R) \cong R$.
In particular, for $R = \FF_p$, $W(R) = \ZZ_p$.

%Although the Witt vector functor is defined for arbitrary rings, it is most often evaluated only on
%perfect rings of characteristic $p$. However, more general rings occur often in applications of Witt vectors
%in arithmetic (e.g., in the study of relative crystalline cohomology as in Langer-Zink \cite{LZ04, LZ05}) and topology
%(e.g., in $K$-theoretic contexts such as that of Hesselholt and Madsen \cite{HM04}).

In this paper, we study the kernel and cokernel of the Frobenius endomorphism on $W(R)$.
In case $p=0$ in $R$,
this map is induced by functoriality from the Frobenius endomorphism of $R$,
and in particular is injective when $R$ is reduced and bijective when $R$ is perfect.
If $p \neq 0$ in $R$, the Frobenius map is somewhat more mysterious; to begin with, it is
never injective. In fact, it is easy and useful to construct many elements of the kernel.
On the other hand, Frobenius is surjective in some cases, although these seem to be somewhat artificial;
the simplest nontrivial example we have found is the ring of integers in a spherical completion of $\overline{\Qp}$.

While surjectivity of Frobenius on full Witt vectors is rather rare, some weaker conditions turn out
to be more commonly satisfied and more relevant to applications. For instance, one can view
the full ring of Witt vectors as an inverse limit of finite-length truncations, and the condition of 
surjectivity of Frobenius on finite levels is satisfied quite often. For instance, 
this holds for $R$ equal to the ring of integers in any infinite algebraic extension of $\QQ$ which is sufficiently ramified at $p$ (e.g., the $p$-cyclotomic extension).
In fact, this condition can be used to give a purely ring-theoretic
formulation of a very strong generalization of Faltings's \emph{almost purity theorem}, based on recent work
by the second author and Liu \cite{KL11} and by Scholze \cite{Sch11}.

One principal motivation for studying the Frobenius on Witt vectors is
to reframe $p$-adic Hodge theory in terms of Witt vectors of characteristic $0$ rings, and ultimately
to globalize the constructions with an eye towards study of global \'etale cohomology, $K$-theory, and $L$-functions.
We will pursue these goals in subsequent papers.

\section{Witt vectors}

Throughout this section, let $R$ denote an arbitrary ring.
For more details on the construction of $p$-typical Witt vectors, see \cite[Section~0.1]{Ill79}.
\begin{Def}
For each nonnegative integer $n$, the ring $W_{p^n}(R)$ is defined to have underlying set
$W_{p^n}(R) := R^{n+1}$ with an exotic ring structure characterized by functoriality in $R$ and the property that for $i=0,\dots,n$, the $p^i$-th \emph{ghost component map} $w_{p^i}: W_{p^n}(R) \to R$ defined by
\[
w_{p^i}(r_1, r_p, \dots, r_{p^n}) = r_1^{p^i} + p r_p^{p^{i-1}} + \cdots + p^i r_{p^i}
\]
is a ring homomorphism. These rings carry \emph{Frobenius homomorphisms}
$F: W_{p^{n+1}}(R) \to W_{p^n}(R)$,  again functorial in $R$, such that for $i=0,\dots,n$, we have $w_{p^i} \circ F = w_{p^{i+1}}$. Moreover, the \emph{Verschiebung} maps $V: W_{p^n}(R) \to W_{p^{n+1}}(R)$ defined by the formula
$V(r_1, \dots, r_{p^n}) = (0, r_1,\dots,r_{p^n})$ are additive (but not multiplicative).

There is a natural restriction map $W_{p^{n+1}}(R) \to W_{p^n}(R)$ obtained by forgetting the last component;
define $W(R)$ to be the inverse limit of the $W_{p^n}(R)$ via these restriction maps.
The Frobenius homomorphisms at finite levels then collate to define another Frobenius homomorphism $F: W(R) \to W(R)$;
there is also a collated Verschiebung map $V: W(R) \to W(R)$.  The ghost component maps also collate to define a ghost map:  $w: W(R) \rightarrow R^{\mathbb{N}}$.  We equip the target with component-wise ring operations; the map $w$ is then a ring homomorphism.

In either $W_{p^n}(R)$ or $W(R)$, an element of the form $(r, 0, 0, \dots)$ is called a \emph{Teichm\"uller element}
and denoted $[r]$. These elements are multiplicative: for all $r_1, r_2 \in R$, $[r_1 r_2] = [r_1] [r_2]$.

Some additional properties of Witt vectors are the following.
\begin{enumerate}
\item[(a)]
For $r \in R$, $F([r]) = [r^p]$.
\item[(b)]
For $\underline{x} \in W_{p^n}(R)$, $(F \circ V)(\underline{x}) = p\underline{x}$.
\item[(c)]
For $\underline{x} \in W_{p^n}(R)$ and $\underline{y} \in W_{p^{n+1}}(R)$, $V(\underline{x} F(\underline{y})) = V(\underline{x}) \underline{y}$.
\item[(d)]
For $\underline{x} \in W_{p^n}(R)$, $\underline{x} = \sum_{i=0}^n V^i([x_{p^i}])$.
\end{enumerate}
\end{Def}

\begin{Rmk} \label{universal identities}
A standard method of proving identities about Witt vectors and their operations
is \emph{reduction to the universal case}: take $R$ to be a polynomial ring in many variables over $\ZZ$,
form Witt vectors whose components are distinct variables, then verify the desired identities at the level of ghost 
components. This suffices because $R$ is now $p$-torsion-free, so the ghost map is injective.
\end{Rmk}

%\begin{Rmk}
%For properties of Witt vectors involving divisibilities by $p$, reduction to the universal case as described above is
%insufficient, because the inverse of the ghost map involves divisions by $p$. Instead, one must supplement with one of two
%other techniques.
%\begin{enumerate}
%\item
%The \emph{Cartier-Dieudonn\'e-Dwork lemma}: if $\varphi: R \to R$ is a ring homomorphism lifting the
%$p$-th power map on $R/pR$, then $w_1, w_p, w_{p^2}, \dots$ forms a sequence of ghost components of an element of $W(R)$
%if and only if $\varphi(w_{p^{n-1}}) \equiv w_{p^n} \mod{p^n}$ for each positive integer $n$.
%(The corresponding statement on finite levels is also true.)
%\item
%The \emph{splitting principle}: for any ring $R$ and any nonnegative integer $n$, there exists a faithfully
%flat ring homomorphism $R \to S$ such that each element of $W_{p^n}(R)$ splits as a sum of Teichm\"uller elements
%of $W_{p^n}(S)$.  
%\end{enumerate}
%\end{Rmk}

We will need a couple of other $p$-divisibility properties.
 We first prove the following lemma.

\begin{Lem} \label{components of p}
In $W(\ZZ/p^2 \ZZ)$, we have
\[
p = (p, (-1)^{p-1}, 0, 0, \dots) = [p] + V([(-1)^{p-1}]).
\]
\end{Lem}
\begin{proof}
Write $p = (x_1, x_p, \dots) \in W(\ZZ)$. Then $x_1 = p$ and 
$x_p = (p - p^p)/p = 1 - p^{p-1}$, which is congruent to 1 mod $p^2$ if $p > 2$ and to $3$ mod 4 if $p=2$.
To complete the argument, we show by induction on $n$ that for each $n \geq 1$, we have $x_{p^i} \equiv 0 \mod{p^2}$ for $2 \leq i \leq n$.
The base case $n=1$ is vacuously true. For the induction step, considering the $p^n$-th ghost component of $p$, write
\[
p^n x_{p^n} = -x_1^{p^n} + p(1 - x_p^{p^{n-1}}) - \sum_{i=2}^{n-1} p^i x_{p^i}^{p^{n-i}}.
\]
To complete the induction, it suffices to check that each term on the right side has $p$-adic valuation at least $n+2$. This is clear for the first term because $p^n \geq n+2$. For the second term, treating $p = 2$ and $p > 2$ separately, we have $x_p \equiv (-1)^{p-1} \mod{p^{p-1}}$ and so $x_p^p \equiv 1 \mod{p^3}$.
We then have $x_p^{p^{n-1}} \equiv 1 \mod{p^{3+n-2}}$, so the second term is indeed divisible by $p^{n+2}$. For the terms in the sum, the claim is again clear because $i+2p^{n-i} \geq i + 2(n-i+1) \geq n+2$.
\end{proof}

\begin{Lem} \label{83111lem}
Take $\underline{x}, \underline{y} \in W(R)$ with $F(\underline{x}) = \underline{y}$. 
\begin{enumerate}
\item[(a)]
For each nonnegative integer $i$, we have $y_{p^i} = x_{p^i}^p + px_{p^{i+1}} + pf_{p^i}(x_1,\ldots,x_{p^{i}})$, where $f_{p^i}$ is a 
certain universal polynomial with coefficients in $\ZZ$ which is homogeneous of degree $p^{i+1}$ under the weighting in which the variable $x_{p^j}$ has weight $p^j$.  
\item[(b)]
For $i \geq 1$, the coefficient of $x_1^{p^{i+1}}$ in $f_{p^i}$ equals $0$.
\item[(c)]
For $i \geq 2$, the coefficient of $x_p^{p^i}$ in $f_{p^i}$ is divisible by $p$.
\item[(d)]
The coefficient of $x_p^p$ in $f_p$ equals $-p^{p-2}$ modulo $p$.
\item[(e)]
For $p=2$ and $i \geq 2$, $f_{p^i}$ belongs to the ideal generated by
$2, x_1, x_p^p-x_{p^2}, x_{p^3}, \dots, x_{p^i}$.
\end{enumerate}
\end{Lem}
\begin{proof}
By reduction to the universal case, we see that 
$y_{p^i}$ equals a universal polynomial in $x_1,\dots,x_{p^{i+1}}$ with coefficients in $\ZZ$ which is 
homogeneous of degree $p^{i+1}$ for the given weighting.
This polynomial is congruent to $x_{p^i}^p$ modulo $p$ by
\cite[(1.3.5)]{Ill79}. Each of the remaining assertions concerns a particular coefficient of this polynomial,
and so may be checked after setting all other variables to 0.

To finish checking (a), we must check that $y_{p^i}$ does not depend on $x_{p^{i+1}}$.
We may assume $x_1 = \cdots = x_{p^i} = 0$, so that $\underline{x} = V^{i+1}([x_{p^{i+1}}])$;
then $F(\underline{x}) = p V^{i}([x_{p^{i+1}}]) = (0,\dots,0,p x_{p^{i+1}})$. 

To check (b), we may assume that $x_p = x_{p^2} = \dots = 0$, so that $\underline{x} = [x_1]$.
In this case, $F(\underline{x}) = [x_1^p]$, so the claim follows.

To check (c), we may assume that $x_1 = x_{p^2} = x_{p^3} = \cdots = 0$, so that $\underline{x} = V([x_p])$. 
In this case, the claim is that $y_{p^i} \equiv 0 \mod{p^2}$  for $i \geq 2$.
Since $F(\underline{x}) = p[x_p]$, by homogeneity it is sufficient to check the claim for $x_p = 1$.
In this case, it follows from Lemma~\ref{components of p}. We may similarly check (d).

To check (e), we may assume that $x_1 = x_{p^3} = x_{p^4} = \cdots = 0$, so that
$\underline{x} = V([x_p]) + V^2([x_{p^2}])$. By homogeneity, it is sufficient to check the claim for
$x_p = x_{p^2} = 1$.
In $W(\ZZ)$, we have $V(1) = 1 + [-1]$ by computation of ghost components, so
$1 + V(1) = 2 + [-1]$. In $W(\ZZ/4\ZZ)$, by Lemma~\ref{components of p} we have
\[
\underline{y} = [2] + 2 V(1) = [2] + V(2) = [2] + V(1 + V(1) - [-1])
= 2 + V(1) + V^2(1).
\]
This implies the desired result.
\end{proof}

\begin{Rmk}
Suppose $R$ is a ring in which $p=0$, and let $\varphi: R \rightarrow R$ denote the Frobenius homomorphism on $R$. 
By Lemma~\ref{83111lem}, we have $F(r_1,r_p,r_{p^2},\ldots) = (r_1^p, r_p^p, r_{p^2}^p,\ldots)$.  
As a result, $F$ is injective/surjective/bijective if and only if $\varphi$ is injective/surjective/bijective.  In particular,
$F$ is injective if and only if $R$ is reduced, and $F$ is bijective if and only if $R$ is perfect.
Similarly, the finite level Frobenius map $F: W_{p^{n}}(R) \rightarrow W_{p^{n-1}}(R)$, which sends $(r_1,\ldots,r_{p^n})$ to $(r_1^p,r_p^p,\ldots,r_{p^{n-1}}^p)$, is injective only if $R=0$, and is surjective
if and only if $\varphi$ is surjective.
\end{Rmk}

%
%\begin{Custom2}
%Set notation as in Lemma~\ref{83111lem}.
%\begin{enumerate}
%\item[(a)]
%For $i \geq 1$, the coefficient of $x_1^{p^{i+1}}$ in $f_{p^i}$ equals $0$.
%\item[(b)]
%For $i \geq 2$, the coefficient of $x_p^{p^i}$ in $f_{p^i}$ is divisible by $p$.
%\item[(c)]
%The coefficient of $x_p^p$ in $f_p$ equals $-p^{p-2}$ modulo $p$.
%\item[(d)]
%For $p=2$ and $i \geq 2$, $f_{p^i}$ belongs to the ideal generated by
%$2, x_1, x_p^p-x_{p^2}, x_{p^3}, \dots, x_{p^i}$.
%\end{enumerate}
%\end{Custom2}
%\begin{proof}
%\end{proof}

\section{The kernel of Frobenius}

When $R$ is a ring not of characteristic $p$, it is easy to see that $F: W(R) \to W(R)$ cannot be injective;
for instance, the Cartier-Dieudonn\'e-Dwork lemma
implies that $p,0,0,\dots$ arises as the sequence of ghost components of some element of $W(\ZZ)$.
More generally, one can determine exactly which elements of $R$ can occur as the first component of an element of
the kernel of $F$. This will be useful in our analysis of surjectivity of $F$. 

\begin{Def} \label{82911def}
Given a ring $R$, define the sets $I_0 = R$ and $I_i := \{r \in R \mid r^p \in pI_{i-1}\}$ for $i>0$;
it is apparent that $I_0 \supseteq I_1 \supseteq I_2 \supseteq \cdots$.
We will see below that each $I_i$ is an ideal. 
Also define $I_\infty = \cap_{i=1}^\infty I_i$, so that $I_\infty = \{r \in R \mid r^p \in p I_\infty\}$.
\end{Def}

\begin{Lem}
For each $i \geq 0$, the set $I_i$ defined above is an ideal.
\end{Lem}

\begin{proof}
We proceed by induction on $i$, the case $i=0$ being obvious. Given that $I_{i-1}$ is an ideal,
it is clear that  $I_i$ is closed under multiplication by arbitrary elements of $R$.  It remains to show that if $x,y \in I_i$, then $x+y \in I_i$.  Using the definition, we must check that $x^p + px^{p-1}y + \cdots + pxy^{p-1} + y^p \in pI_{i-1}$.  That $x^p, y^p \in pI_{i-1}$ follows from $x,y \in I_i$.  That the remaining terms are in $pI_{i-1}$ follows from $x,y \in I_i \subseteq I_{i-1}$.
\end{proof}

\begin{Rmk}
If $R$ is the ring of integers in an algebraic closure of $\Qp$ (or the completion thereof), then $I_i$ is the 
principal ideal generated by any element of valuation $\frac{1}{p} + \frac{1}{p^2} + \cdots + \frac{1}{p^i}$,
while $I_\infty$ is the principal ideal generated by any element of valuation $\frac{1}{p-1}$.
\end{Rmk}

\begin{Def} \label{BallDef}
For any ring $R$, any $r_0 \in R$, and any $i \geq 0$ (including $i = \infty$), 
define $B(r_0,I_i) := r_0 + I_i = \{ r \in R \mid r - r_0 \in I_i \}$.
The notation is meant to suggest that $B$ is a \emph{ball} centered at $r_0$.
\end{Def}

The significance of the ideals $I_i$ is the following.

\begin{Prop} \label{82411kerprop}
Let $R$ be a ring, let $i$ be a positive integer, and let $n$ be either $\infty$ or an integer greater than or equal to $i$.
\begin{enumerate}
\item[(a)]
If $\underline{x}, \underline{y} \in W_{p^i}(R)$ are such that 
$x_{p^j} - y_{p^j} \in I_{n-j}$ for $j=0,\dots,i$, then for
$\underline{x}' := F(\underline{x})$, $\underline{y}' := F(\underline{y})$, we have
$x'_{p^j} - y'_{p^j} \in pI_{n-j  - 1}$ for $j=0,\dots,i-1$.
\item[(b)]
Take $\underline{x}, \underline{y} \in W_{p^i}(R)$ and put $\underline{x}' := F(\underline{x})$, $\underline{y}' := F(\underline{y})$.
If $x'_{p^j} - y'_{p^j} \in pI_{n-j}$ for $j=0,\dots,i$
and $x_{p^i} - y_{p^i} \in I_{n-i}$, then $x_{p^j} - y_{p^j} \in I_{n-j}$ for $j=0,\dots,i$.
In particular, if $\underline{x}' = \underline{y}'$, then this always holds when $n = i$.
\item[(c)]
Choose $x_1,\dots,x_{p^{i-1}}, y_1, \dots, y_{p^i} \in R$ with $x_{p^j} - y_{p^j} \in I_{n-j}$ 
for $j=0,\dots,i-1$. If $i>1$, assume also that 
$F(x_1, x_p,\dots,x_{p^{i-1}}) = F(y_1,y_p,\dots,y_{p^{i-1}})$.
Then there exists $x_{p^{i}} \in B(y_{p^i}, I_{n-i})$ such that 
$F(x_1, x_p,\dots,x_{p^{i}}) = F(y_1,y_p,\dots,y_{p^{i}})$.
\item[(d)]
For any $\underline{x} \in W(R)$ for which $x_{p^i} \in p I_\infty$ for all $i$,
there exists $\underline{y} \in W(R)$ for which $y_1 = 0$, $y_{p^i} \in I_\infty$ for all $i$,
and $F(\underline{y}) = \underline{x}$.
\end{enumerate}
\end{Prop}
\begin{proof}
To check (a), apply Lemma~\ref{83111lem} to write
$x'_{p^j}-y'_{p^j} = x_{p^j}^p - y_{p^j}^p + p(x_{p^{j+1}} - y_{p^{j+1}})
+ p(f_{p^j}(x_1,\dots,x_{p^j}) - f_{p^j}(y_1,\dots,y_{p^j}))$.  Writing $y_{p^j} = x_{p^j} - (x_{p^j} - y_{p^j})$, we
note that $x_{p^j}^p - y_{p^j}^p$ belongs to the ideal generated by $(x_{p^j} - y_{p^j})^p$ and $p(x_{p^j} - y_{p^j})$. 
Note also that $p(f_{p^j}(x_1,\dots,x_{p^j}) - f_{p^j}(y_1,\dots,y_{p^j}))$ belongs to the ideal generated by $p(x_1 - y_1), \dots, p(x_{p^{j}} - y_{p^{j}})$.  
It follows that $x'_{p^j}-y'_{p^j} \in p I_{n-j  - 1}$.

To check (b), we first check that under the hypotheses of (b), if there exists $0 \leq k \leq  n-i+1$ such that
$x_{p^j} - y_{p^j} \in I_{k}$ for $j=0,\dots,i$, then $x_{p^{j}} - y_{p^{j}} \in I_{k+1}$
for $j=0,\dots,i-1$. For $j \in \{0,\dots,i-1\}$, apply Lemma~\ref{83111lem} to write
\begin{align*}
(x_{p^{j}} - y_{p^{j}})^p - (x'_{p^j}-y'_{p^j}) &=
((x_{p^{j}} - y_{p^{j}})^p - x_{p^{j}}^p + y_{p^{j}}^p) \\
&-p (x_{p^{j+1}}-y_{p^{j+1}} + f_{p^{j}}(x_1,\dots,x_{p^{j}})
- f_{p^{j}}(y_1,\dots,y_{p^{j}})).
\end{align*}
{}From this equality, we see that $(x_{p^{j}} - y_{p^{j}})^p - (x'_{p^j}-y'_{p^j})$
belongs to the ideal generated by $p(x_1 - y_1), \dots, p(x_{p^{j+1}} - y_{p^{j+1}})$. 
This ideal is contained in $p I_k$ by hypothesis.  By assumption we also have $x'_{p^j}-y'_{p^j} \in pI_{n-j}$, and because $n - j \geq n - i + 1 \geq k$, we have $x'_{p^j}-y'_{p^j} \in pI_k$ as well.  Hence $(x_{p^j} - y_{p^j})^p \in pI_k$,  and
so we have $x_{p^{j}} - y_{p^{j}} \in I_{k+1}$ as claimed.

Note that the hypothesis of the previous paragraph is always satisfied for $k=0$ because $I_0 = R$.  The previous paragraph gives us control over the terms $x_1 - y_1, \ldots, x_{p^{i-1}}-y_{p^{i-1}}.$
Since $x_{p^i} - y_{p^i} \in I_{n-i}$ by assumption, we may induct on $k$ to deduce that 
$x_{p^j} - y_{p^j} \in I_{n-i+1}$ for $j=0,\dots,i-1$. In particular, $x_{p^{i-1}} - y_{p^{i-1}} \in I_{n-i+1}$;
we may now induct on $i$ to deduce (b).

To check (c), note that by Lemma~\ref{83111lem} again, it is sufficient to find $x_{p^i} \in B(y_{p^i}, I_{n-i})$
such that $x_{p^{i-1}}^p + p x_{p^{i}} + p f_{p^{i-1}}(x_1,\dots,x_{p^{i-1}}) = y_{p^{i-1}}^p + p y_{p^{i}} + p f_{p^{i-1}}(y_1,\dots,y_{p^{i-1}})$.
This is possible because $x_{p^{i-1}} - y_{p^{i-1}} \in I_{n-i+1}$ and $x_1-y_1,\dots,x_{p^{i-1}}-y_{p^{i-1}} \in I_{n-i}$, so
as in the proof of (a) we have
$x_{p^{i-1}}^p - y_{p^{i-1}}^p + p (f_{p^{i-1}}(x_1,\dots,x_{p^{i-1}}) -f_{p^{i-1}}(y_1,\dots,y_{p^{i-1}})) 
\in p I_{n-i}$.

To check (d), we construct the $y_{p^i}$ recursively, choosing $y_1 = 0$. Given
$y_1,\dots,y_{p^i}$, we must choose
$y_{p^{i+1}}$ so that in the notation of Lemma~\ref{83111lem}, we have
$y_{p^i}^p + p y_{p^{i+1}} + p f_{p^i}(y_1,\dots,y_{p^i}) = x_{p^i}$.
This is possible because $y_{p^i}^p$, $p f_{p^i}(y_1,\dots,y_{p^i})$, and $x_{p^i}$ all belong to
$p I_\infty$.
\end{proof}

\begin{Cor}
Let $R$ be a ring and let $n$ be either $\infty$ or a positive integer.
Then an element $r \in R$ occurs as the first component of an element
of the kernel of $F:W_{p^{n}}(R) \to W_{p^{n-1}}(R)$ if and only if $r \in I_n$.
\end{Cor}
\begin{proof}
Suppose that $n < \infty$.
If $r = z_1$ for $\underline{z} \in W_{p^{n}}(R)$ such that $F(\underline{z}) = 0$,
then trivially $z_{p^n} \in I_0$. By Proposition~\ref{82411kerprop}(b),
$z_1 \in I_n$; the same conclusion holds for $n = \infty$.
Conversely, suppose $r \in I_n$. Put $z_1 = r$.
By Proposition~\ref{82411kerprop}(c) applied repeatedly,
for each positive integer $i \leq n$, we can find $z_{p^i} \in I_{n-i}$ so that
$F(z_1, z_p, \dots, z_{p^i}) = 0$. This proves the claim.
\end{proof}

\begin{Rmk}
The image under the ghost map of any element in the kernel of $F$ has the form $(*,0,0,\ldots)$.
If $R$ is a $p$-torsion-free ring, the ghost map is injective, so any element of the kernel of $F$ is uniquely determined by its first
component. In this case, we may combine Proposition~\ref{82411kerprop}(b) and (c) to deduce that
if $\underline{z} \in W_{p^i}(R)$ is such that $F(\underline{z}) = 0$
and $z_{1} \in I_{n}$ for some $n \geq i$, then $z_{p^j} \in I_{n-j}$ for $j=0,\dots,i$.
\end{Rmk}

\section{Surjectivity conditions} 

Surjectivity of the Witt vector Frobenius turns out to be a subtler property than injectivity, because
there are many partial forms of surjectivity which occur much more frequently than full surjectivity.
We first list a number of such conditions, then identify logical relationships among them.

\begin{Def}
For $R$ an arbitrary ring, label the conditions on $R$ as follows.
\begin{enumerate}[label=(\roman*)]
\item \label{surj} $F: W(R) \rightarrow W(R)$ is surjective.
\item \label{finlevsurj} $F: W_{p^n}(R) \rightarrow W_{p^{n-1}}(R)$ is surjective for all $n \geq 2$.
\item[\ref{finlevsurj}$'$] $F: W_{p^2}(R) \rightarrow W_p(R)$ is surjective.
\item \label{Teichdense} For every $\underline{x} \in W(R)$, there exists $r \in R$ such that $\underline{x} - [r] \in pW(R)$.
\item \label{Teichinimage} The image of $F: W(R) \rightarrow W(R)$ contains all Teichm\"uller elements $[r]$.
\item \label{Teichinimagefin} $F: W_{p^n}(R) \rightarrow W_{p^{n-1}}(R)$ contains all elements of the form $(r,0,\ldots,0)$ for all $n \geq 2$.
\item[\ref{Teichinimagefin}$'$] $F: W_{p^2}(R) \rightarrow W_p(R)$ contains all elements of the form $(r,0)$.
\item \label{Vinimage} The image of $F: W(R) \rightarrow W(R)$ contains $V(1)$.
\item \label{Vinimageall} For all $n \geq 2$, the image of $F: W_{p^{n}}(R) \rightarrow W_{p^{n-1}}(R)$ contains $V(1)$.
\item \label{Vinimagefin2} For all $n \geq 2$, the image of $F: W_{p^{n}}(R) \rightarrow W_{p^{n-1}}(R)$ contains $V^{n-1}(1)$. \item \label{Vinimagefin} The image of $F: W_{p^2}(R) \rightarrow W_p(R)$ contains $V(1)$. 
\item \label{lev1surj} $F^n: W_{p^n}(R) \rightarrow W_1(R)$ is surjective for all $n \geq 1$.
\item[\ref{lev1surj}$'$] $F: W_{p}(R) \rightarrow W_1(R)$ is surjective.
\item \label{pinvertible} $R$ contains $p^{-1}$.
\item \label{sphericallycomp} For any $r_0, r_1, \dots \in R$ such that $B(r_0, I_0) \supseteq B(r_1, I_1) \supseteq \cdots$
(in the notation of Definition \ref{BallDef}), the intersection $\cap_{i=0}^\infty B(r_i,I_i)$ is non-empty.
\item \label{pthroots} The $p$-th power map on $R/p I_\infty$ (which need not be a ring homomorphism)  is surjective.
\item \label{pthrootsmodfinite} For each $n \geq 1$, the $p$-th power map on $R/p I_n$ is surjective.
\item[\ref{pthrootsmodfinite}$'$] The $p$-th power map on $R/p I_1$ is surjective.
\item \label{prmodp^2} For every $r \in R$, there exists $s \in R$ such that $s^p \equiv pr \mod p^2R$.
\item \label{pmodp^2new} There exist $r,s$ in $R$ such that $r^p \equiv -p \mod ps R$ and $s \in I_1$.
\item \label{some power} There exist $r,s$ in $R$ such that $r^p \equiv -p \mod ps R$ and $s^N \in pR$ for some integer $N > 0$.
\item \label{pthrootsmodp} The Frobenius homomorphism $\varphi: \overline{r} \mapsto \overline{r}^p$ on $R/pR$ is surjective.
\end{enumerate}
\end{Def}

These conditions are represented graphically  in Figure~\ref{logical implications}.
Note that  conditions in the top left quadrant refer to infinite Witt vectors, conditions in
the top right quadrant refer to finite Witt vectors, and conditions below the dashed line
refer to $R$ itself.

\begin{figure}
\[ \hspace{-.6in} \xymatrix@R=.20in@C=.20in@M=0in{
%%%here are the nodes%%%
& & & & & &\ar@{--}[dddd] & & & & &  &  & \\
& &  &\text{\ref{surj}} \ar@{}="surj" &  &
\text{\ref{Teichinimage}} \ar@{}="Teichinimage" & &
\text{\ref{Teichinimagefin}} \ar@{}="Teichinimagefin"&  &  &  &  &  &
\text{\ref{finlevsurj}} \ar@{}="finlevsurj" \\
& &  &  &  &  &  &  &  &  &  &  &  & \\
& &  & \text{\ref{Teichdense}} \ar@{}="Teichdense"&  &
\text{\ref{Vinimage}} \ar@{}="Vinimage"&  &\text{\ref{Vinimageall}} \ar@{}="Vinimageall"&  &
\text{\ref{Vinimagefin2}} \ar@{}="Vinimagefin2" &  & 
\text{\ref{Vinimagefin}} \ar@{}="Vinimagefin"&  & &  &
\text{\ref{lev1surj}} \ar@{}="lev1surj" &  \\
\ar@{--}[rrrrrrrrrrrrrrrr] & &  &  & & &  &  &  & &  &  &  &  &  &  & & \\
& \text{\ref{pinvertible}} \ar@{}="pinvertible"&  & 
\text{\ref{sphericallycomp}} \ar@{}="sphericallycomp" &   & 
\text{\ref{pthroots}} \ar@{}="pthroots"&& 
\text{\ref{pthrootsmodfinite}} \ar@{}="pthrootsmodfinite" &  &
\text{\ref{prmodp^2}} \ar@{}="prmodp^2"&  &
\text{\ref{pmodp^2new}} \ar@{}="pmodp^2new"&  &
\text{\ref{some power}} \ar@{}="some power"&  &
\text{\ref{pthrootsmodp}} \ar@{}="pthrootsmodp"
%%%here are the arrows%%%
\ar  "surj" ; "Teichinimage"
\ar  "surj" ; "Teichdense"
\ar@(dl,ul)  "surj" ; "sphericallycomp"
\ar@{<->}  "finlevsurj" ; "Teichinimagefin"
%\ar@(l,ur)  "finlevsurj" ; "Vinimageall"
%\ar@(l,ur)  "finlevsurj" ; "Vinimagefin2"
\ar "finlevsurj" ; "lev1surj"
\ar "pthrootsmodfinite" ; "prmodp^2"
\ar "pmodp^2new" ; "some power"
\ar@(dl,ul)@{<->} "Teichinimagefin" ; "pthrootsmodfinite"
\ar "Teichinimage" ; "Teichinimagefin"
\ar "Vinimage" ; "Vinimageall"
\ar@(ur,ul) "Vinimageall" ; "Vinimagefin"
\ar "Vinimagefin2" ; "Vinimagefin"
\ar@{<->} "Vinimagefin" ; "pmodp^2new"
\ar@{<->} "lev1surj" ; "pthrootsmodp"
\ar "prmodp^2" ; "Vinimagefin2"
\ar "prmodp^2" ; "Vinimage"
\ar@/^1.5pc/@{<->} "pthroots" ; "Teichinimage"
\ar@/^1pc/@{.>} "lev1surj" ; "finlevsurj"
\ar@{.>} "some power" ; "finlevsurj"
\ar@(u,ul) "pinvertible" ; "surj"
\ar@/_1pc/@{=>} "Teichdense" ; "surj" 
\ar@/^1pc/@{=>} "Teichinimage" ; "surj" 
\ar@(l, l)@{:>} "sphericallycomp" ; "surj" 
\ar@(ul,ur)@{:>} "finlevsurj" ; "surj"
%%%
}
\]
\caption{Logical implications among conditions on the ring $R$.}
 \label{logical implications}
\end{figure}

\begin{Thm} \label{theorem on logical implications}
For any ring $R$, we have $\ref{finlevsurj} \Leftrightarrow \ref{finlevsurj}'$, $\ref{Teichinimagefin} \Leftrightarrow \ref{Teichinimagefin}'$,
$\ref{lev1surj} \Leftrightarrow \ref{lev1surj}'$, 
and $\ref{pthrootsmodfinite} \Leftrightarrow \ref{pthrootsmodfinite}'$.
In addition, each solid single arrow in Figure~\ref{logical implications} represents a direct implication,
and for each other arrow type, the conditions at the tails of the arrows of that type together imply the condition at the target.
\end{Thm}

The proof of Theorem~\ref{theorem on logical implications} will occupy the rest of this section.
First, however, we mention some consequences of Theorem~\ref{theorem on logical implications},
and some negative results which follow from some examples considered in 
Section~\ref{eg section}.
\begin{Cor}
For any ring $R$, we have the following equivalences.
\begin{itemize}
\item $\ref{surj} \Leftrightarrow \ref{finlevsurj} + \ref{sphericallycomp} \Leftrightarrow \ref{Teichdense} + \ref{Teichinimage} \Leftrightarrow \ref{Teichdense} + \ref{pthroots} $
\item $\ref{finlevsurj} \Leftrightarrow \ref{Teichinimagefin} \Leftrightarrow 
\ref{pthrootsmodfinite} \Leftrightarrow \Bigg\{\ref{lev1surj} \text{ or } \ref{pthrootsmodp}\Bigg\}  + \Bigg\{\parbox{8pc}{\ref{Vinimage}, \ref{Vinimageall}, \ref{Vinimagefin2}, \ref{Vinimagefin}, \ref{prmodp^2}, \ref{pmodp^2new}, \text{ or } \ref{some power}} \Bigg\} $
\end{itemize}
\end{Cor}

\begin{Rmk}
The following implications fail to hold by virtue of the indicated examples.
\begin{itemize}
\item $\ref{surj} \not\Rightarrow \ref{pinvertible}$ by Example \ref{sphereeg}.
\item $\ref{finlevsurj} \not\Rightarrow \ref{surj}$ by Example \ref{OCpeg} (or from $\ref{finlevsurj} \not\Rightarrow \ref{Teichinimage}$ below).
\item $\ref{finlevsurj} \not\Rightarrow \ref{Teichdense}$ by Example \ref{OCpeg}.
\item $\ref{finlevsurj} \not\Rightarrow \ref{Teichinimage}$ by Example \ref{mupinftyeg}.
\item $\ref{finlevsurj} \not\Rightarrow \ref{sphericallycomp}$ by Example \ref{OCpeg}.
\item $\ref{Teichinimage} \not\Rightarrow \ref{surj}$ by Example \ref{OCpeg}.
\item $\ref{Vinimage} \not\Rightarrow \ref{prmodp^2}$ by Example \ref{mup2eg}.
\item $\ref{Vinimage} \not\Rightarrow \ref{pthrootsmodp}$ by Example \ref{mup2eg}.
\item $\ref{sphericallycomp} \not\Rightarrow \ref{some power}$ by Example \ref{ZZeg}.
\item $\ref{prmodp^2} \not\Rightarrow \ref{pthrootsmodp}$ by Example \ref{FpTeg}.
\item $\ref{pthrootsmodp} \not\Rightarrow \ref{some power}$ by Example \ref{ZZeg}.
\end{itemize}
\end{Rmk}

\begin{Rmk}
It seems that there should be some relationship between \ref{Teichdense} and \ref{sphericallycomp},
but we were unable to clarify this.
\end{Rmk}

\noindent \textbf{Proof of Theorem~\ref{theorem on logical implications}.}  We now prove the implications represented in 
Figure~\ref{logical implications}.

\begin{itemize}
\item $\ref{surj} \Rightarrow \ref{Teichinimage}$; $\ref{finlevsurj} \Rightarrow \ref{finlevsurj}'$; $\ref{finlevsurj} \Rightarrow \ref{Teichinimagefin}$; $\ref{finlevsurj}' \Rightarrow \ref{Teichinimagefin}'$;  $\ref{Teichinimagefin} \Rightarrow \ref{Teichinimagefin}'$;  $\ref{Vinimageall} \Rightarrow \ref{Vinimagefin}$; $\ref{Vinimagefin2} \Rightarrow \ref{Vinimagefin}$;  $\ref{lev1surj} \Rightarrow \ref{lev1surj}'$;
$\ref{pthrootsmodfinite} \Rightarrow \ref{pthrootsmodfinite}'$;
 $\ref{pmodp^2new} \Rightarrow \ref{some power}; \ref{Teichinimage} \Rightarrow \ref{Teichinimagefin}; \ref{Vinimage} \Rightarrow \ref{Vinimageall}; \ref{finlevsurj} \Rightarrow \ref{lev1surj}$
\begin{proof}
These are all obvious.
\end{proof}
\item $ \ref{surj} \Rightarrow \ref{Teichdense}$
\begin{proof}
Let $\underline{x} \in W(R)$ denote an arbitrary element.  We may write $\underline{x} = \sum V^i([x_{p^i}])$, and because $F \circ V = p$, we have $F(\underline{x}) \equiv [x_1^p] \mod pW(R)$.   Since we are assuming that $F$ is surjective, we deduce \ref{Teichdense}.
\end{proof}
\item $ \ref{surj} \Rightarrow \ref{sphericallycomp}$
\begin{proof}
Fix elements $r_i$ as in condition \ref{sphericallycomp}.  Our strategy is to define an element $\underline{y} \in W(R)$ in a special way so that if $\underline{x} \in W(R)$ is such that $F(\underline{x}) =  \underline{y}$, then we must have $x_1 \in \cap_{i=0}^{\infty} B(r_i, I_i)$.  To prescribe our element $\underline{y} \in W(R)$, it suffices to define compatible finite length Witt vectors $\underline{y}^{(p^i)} \in W_{p^i}(R)$ for every $i$.   

Define $\underline{x}^{(p)} \in W_p(R)$ by $\underline{x}^{(p)} = (r_1,0)$ (the second component does not matter).  Set $\underline{y}^{(1)} := F(\underline{x}^{(p)})$.   Now inductively assume we have defined $\underline{x}^{(p^i)} \in W_{p^i}(R)$ for some $i \geq 1$ and with first component $x^{(p^i)}_1 = r_i$.  By Proposition \ref{82411kerprop}(c), we can find an element $\underline{z}^{(p^i)} \in W_{p^i}(R)$ with $z^{(p^i)}_1 = r_{i+1} - r_i \in I_i$ and with $F(\underline{z}^{(p^i)}) = 0$.  Then let $\underline{x}^{(p^{i+1})} \in W_{p^{i+1}}(R)$ denote any element which restricts to $\underline{x}^{(p^i)} + \underline{z}^{(p^i)} \in W_{p^{i}}(R)$.  Set $\underline{y}^{(p^i)} = F(\underline{x}^{(p^i)})$.  Then by construction our elements $\underline{y}^{(p^i)}$ correspond to an element of $\varprojlim W_{p^{i}}(R) \cong W(R)$, which we call $\underline{y}$.

By \ref{surj}, we can find an element $\underline{x}$ such that $F(\underline{x}) = \underline{y}$.  Because $F(\underline{x})$ and $F(\underline{x}^{(p^{i+1})})$ have the same initial $i+1$ components, we have that $x_1 \equiv \underline{x}^{(p^{i+1})}_1 \mod I_{i+1}$.  Because $x_1$ does not depend on $i$, and $x_1^{(p^{i+1})} = r_{i+1}$, we have that $x_1 \in \cap_{i=0}^{\infty} B(r_{i+1},I_{i+1})$, as desired.  
\end{proof}

\item $ \ref{finlevsurj} + \ref{sphericallycomp} \Rightarrow \ref{surj}$
\begin{proof}
Choose any $\underline{y} \in W(R)$. We will construct $\underline{x} \in W(R)$ such that $F(\underline{x}) = \underline{y}$.  We use \ref{finlevsurj} to find elements $\underline{x}^{(1)}, \underline{x}^{(p)}, \ldots  \in W(R)$ so that $F(\underline{x}^{(1)}) = (y_1, *, *,\cdots), F(\underline{x}^{(p)}) = (y_1,y_p,*,*,\cdots),$ and so on. By \ref{sphericallycomp} and Proposition~\ref{82411kerprop}(b),
we may choose $\widetilde{x_{p^j}}$ in the intersection $B_0(x^{(p^j)}_{p^j}, I_0) \cap B_1(x^{(p^{j+1)})}_{p^j}, I_1) \cap \cdots$.
Put $\underline{\widetilde{y}} := F(\widetilde{x_1},\widetilde{x_{p}}, \dots)$.
We first apply Proposition~\ref{82411kerprop}(a) to $(\widetilde{x_1},\dots,\widetilde{x_{p^{i+1}}})$
and $(x^{(p^k)}_1,\dots,x^{(p^k)}_{p^{i+1}})$ for fixed $i$ and increasing $k$, which implies that
$\widetilde{y_{p^i}} - y_{p^i} \in pI_\infty$ for each nonnegative integer $i$.
This means that $\underline{y}$ and $\widetilde{y}$ have the same image in $W(R/pI_\infty)$,
so the difference $\underline{z} = \underline{y} - \underline{\widetilde{y}}$ has all of its components in $pI_\infty$.
By Proposition~\ref{82411kerprop}(d), $\underline{z}$ is in the image of $F$, as then is $\underline{y}$.
\end{proof}
\item $\ref{Teichdense}  + \ref{Teichinimage}\Rightarrow \ref{surj}$
\begin{proof}
This is obvious, given that any element $p\underline{x}' = F(V(\underline{x}'))$ is in the image of Frobenius.
\end{proof}
\item $\ref{Teichinimage} \Rightarrow \ref{pthroots}$; $\ref{Teichinimagefin} \Rightarrow \ref{pthrootsmodfinite}$; $\ref{Teichinimagefin}' \Rightarrow \ref{pthrootsmodfinite}'$
\begin{proof}
Suppose that $n \geq 2$ and that $\underline{x} \in W_{p^{n}}(R)$ and $r \in R$
satisfy $F(\underline{x}) = [r]$.
For each of $k = 0, \dots, n-1$, we check that $x_p, x_{p^2}, \ldots, x_{p^{n-k}}$ belong to $I_k$. This is clear for $k = 0$.
Given the claim for some $k < n-1$, for $i = 1, \dots, n-1-k$ we may apply Lemma~\ref{83111lem} to deduce that
$x_{p^i}^p + p x_{p^{i+1}} + p f_{p^i}(x_1, ..., x_{p^i}) = 0$.
By Lemma~\ref{83111lem}$'$, $f_{p^i}$ contains no pure power of $x_1^{p^{i+1}}$, so
$f_{p^i}(x_1, \dots, x_{p^i})$ belongs to the ideal generated by $x_p, \dots, x_{p^i}$.
Therefore $-p x_{p^{i+1}}$ and $-p f_{p^i}(x_1, \dots, x_{p^i})$ both belong to $pI_k$,
and so $x_{p^i}$ belongs to $I_{k+1}$.
This completes the proof; as a corollary, we observe that $x_p \in I_{n-1}$, and so $r - x_1^p = px_p \in pI_{n-1}$.  The stated implications now follow.
\end{proof}

\item $ \ref{Vinimagefin} \Rightarrow \ref{pmodp^2new}$
\begin{proof}
We are assuming that we can find $\underline{x}$ such that $F(\underline{x}) = V([1])$.  Then the ghost components of $\underline{x}$ must be $(*,0,p)$.  In other words, $x_1^p + px_p = 0$ and $x_1^{p^2} + px_p^p + p^2 x_{p^2} = p$.  The first equality tells us that $x_1^p \in pR$ (and hence $x_1^{p^2} \in p^pR$).  The second equality now tells us $px_p^p \equiv p \mod p^2R$ and so $x_p^p \equiv 1 \mod pR$.  Write $x_p = 1 + s$.  (We are not yet making any claims on $s$ except that $s \in R$.)  Raising both sides to the $p$-th power, we have $x_p^p = 1 + s^p + t$, where $t \in pR$ by the binomial theorem.  On the other hand, we decided above that $x_p^p - 1 \in pR$ as well, so in turn we have $s^p \in pR$.  Returning to the $p$-th ghost component equation $x_1^p + px_p = 0$, we now have $x_1^p \equiv -p \mod psR$ where $s^p \in pR$, as required.  
\end{proof}
\item $ \ref{lev1surj}' \Rightarrow \ref{pthrootsmodp}$; $\ref{pinvertible} \Rightarrow \ref{surj}$
\begin{proof}
Working with ghost components as above, these are obvious.
\end{proof}

\item $\ref{pthroots} \Rightarrow \ref{Teichinimage}$
\begin{proof}
Given $r \in R$, by \ref{pthroots} we may choose $x_1 \in R$, $x_p \in I_\infty$ for which $r = x_1^p + p x_p$.
We now show that we can choose $x_{p^2}, x_{p^3}, \dots \in I_\infty$ so that $F(x_1, \dots, x_{p^n}) = (r, 0, \dots, 0)$
for each $n \geq 1$. 

Given $x_1, \dots, x_{p^n}$, define $f_{p^n}$ as in Lemma~\ref{83111lem}.
By Lemma~\ref{83111lem}$'$, $f_{p^n}$ contains no pure power of $x_1^{p^{n+1}}$, so
$f_{p^n}(x_1, \dots, x_{p^n})$ belongs to the ideal generated by $x_p, \dots, x_{p^n}$, which by construction
is contained in $I_\infty$. It follows that $- x_{p^n}^p - f_{p^n}(x_1, \dots, x_{p^n}) \in p I_\infty$,
so we can find $x_{p^{n+1}} \in pI_\infty$ for which $x_{p^n}^p + p x_{p^{n+1}} + f_{p^n}(x_1, \dots, x_{p^n}) = 0$.
By Lemma~\ref{83111lem}, this choice of $x_{p^{n+1}}$ has the desired effect.
\end{proof}
\item $\ref{prmodp^2} \Rightarrow \ref{Vinimage}$
\begin{proof}
We wish to produce elements $x_1,x_p,\ldots$ of $R$ such that $F(x_1,x_p,\ldots) = (0,1,0,0,\ldots) = V(1)$.  
Using \ref{prmodp^2}, choose $r$ so that $r^p \equiv -p \mod p^2$.  Set $x_1 := r$.  Then clearly we can choose $x_p \equiv 1 \mod p$ such that $F(x_1,x_p) = (0)$. 
Next, in the notation of Lemma \ref{83111lem}, we wish to choose $x_{p^2}$ so that  \[x_p^p + px_{p^2} + pf_p(x_1,x_p)  = 1.\] We also wish to ensure that if $p > 2$, then $x_{p^2} \equiv 0 \mod{p}$, while if $p = 2$, then
$x_{p^2} = 1 \mod{p}$.
To see that this is possible,
we first observe that $x_p^p \equiv 1 \mod{p^2}$. We then note that $f_p(x_1,x_p)$ consists of an element of the ideal
generated by $x_1^p$ (which is a multiple of $p$) plus some constant times $x_p^p$. By Lemma \ref{83111lem}$'$,
if $p > 2$ this constant is divisible by $p$, so $p f_p(x_1,x_p) \equiv 0 \mod{p^2}$. If $p=2$,
this constant is $-1 \mod{2}$, so $p f_p(x_1, x_p) \equiv -2 \pmod{p^2}$. In either case, we obtain $x_{p^2}$ of the desired form.

Now assume that for some $i \geq 2$, we have found $x_1,x_p,\ldots,x_{p^i}$ such that $x_{p^j} \equiv 0 \mod p$ for $j \geq 3$ and such that $F(x_1,x_p,\ldots,x_{p^i}) = (0,1,0,\ldots,0)$.  We then claim that we can find $x_{p^{i+1}} \equiv 0 \mod p$ such that $F(x_1,x_p,\ldots,x_{p^{i+1}}) = (0,(-1)^{p-1},0,\ldots,0)$.  We wish to find $x_{p^{i+1}} \equiv 0 \mod p$ such that 
\[x_{p^i}^p + px_{p^{i+1}} + pf_{p^{i}}(x_1,\ldots,x_p) = 0 \]  with $f_{p^i}(x_1,\ldots,x_p) \equiv 0 \mod p$.  This again follows by Lemma \ref{83111lem}$'$.
\end{proof}
\item $ \ref{prmodp^2}\Rightarrow \ref{Vinimagefin2}$
\begin{proof}
Our goal is to find an element $\underline{x} = (x_1,x_p,\ldots,x_{p^{n}})$ such that $F(\underline{x}) = V^{n-1}(1).$  Ignoring $x_{p^{n}}$ temporarily, we will first find preliminary values for $x_{p^{n-1}},\ldots,x_1$ (in that order), then we will find the actual values for $x_1,\ldots, x_{p^{n}}$ (in that order).  We will write the preliminary values as $\widetilde{x_{p^i}}$.  

Set $\widetilde{x_{p^{n-1}}} = 1$, and then find $\widetilde{x_{p^{n-2}}},\ldots,\widetilde{x_{1}}$ (in that order) such that $\widetilde{x_{p^i}}^p \equiv -p\widetilde{x_{p^{i+1}}} \mod p^2R$.  This is possible by \ref{prmodp^2}.  Note that $\widetilde{x_{p^i}}^p \in pR$ for $0 \leq i \leq n-2$.
Now we will find the actual values $x_1, \ldots, x_{p^{n}}$.  Set $x_1 := \widetilde{x_1}$.   Assume we have found $x_1, \ldots, x_{p^i}$ with $x_{p^j} \equiv \widetilde{x_{p^j}} \mod pR$ for some $i \leq n-2$.  Using the notation of Lemma \ref{83111lem}, we must choose $x_{p^{i+1}}$ such that $x_{p^{i}}^p + px_{p^{i+1}} + pf_{p^{i}}(x_1,\ldots,x_{p^i}) = 0$. Write $x_{p^{i+1}} = \widetilde{x_{p^{i+1}}} + p y_{p^{i+1}}$.  We must choose $y_{p^{i+1}}$ so that 
\[
x_{p^i}^p + p \widetilde{x_{p^{i+1}}} + p^2 y_{p^{i+1}} + pf_{p^i}(x_1,\ldots,x_{p^i}) = 0.
\]
Because $\widetilde{x_{p^i}}^p + p\widetilde{x_{p^{i+1}}} \equiv 0 \mod p^2$ and $\widetilde{x_{p^i}} \equiv x_{p^i} \mod p$, we have that 
$x_{p^i}^p + p \widetilde{x_{p^{i+1}}} \equiv 0 \mod p^2.$  We further have that
$pf_{p^i}(x_1,\ldots,x_{p^i}) \equiv 0 \mod p^2;$
this follows from the homogeneity result in Lemma \ref{83111lem} and the fact that $x_{p^j}^p \equiv 0 \mod p$ for all $j$.  This shows that we can find the required $y_{p^{i+1}}$.

In this way we can construct the components $x_1,\ldots,x_{p^{n-1}}$. Finding the last component $x_{p^{n}}$ is a little different, because the last component of $V^{n-1}(1)$ is $1$ instead of $0$.   This means that we need \[x_{p^{n-1}}^p + px_{p^{n}} + pf_{p^{n-1}}(x_1,\ldots,x_{p^{n-1}}) = 1. \]   But this is easy, because we know $x_{p^{n-1}} \equiv 1 \mod p$.
\end{proof}
\item $\ref{pmodp^2new} \Rightarrow \ref{Vinimagefin}$
\begin{proof}
By Lemma \ref{83111lem}, we must find $x_1,x_p,x_{p^2}$ such that $x_1^p + px_p = 0$ and $x_p^p + px_{p^2} + pf(x_1,x_2) = 1$.
By \ref{pmodp^2new}, we can find an element $x_1$ such that $x_1^p + p + psr = 0$ where $s^p \in pR$.  Thus we choose that element for $x_1$, and we choose $x_p = 1+sr$.  It's then clear that $x_p^p + pf(x_1,x_2) \equiv 1 \mod pR$, and so we can find $x_{p^2}$ forcing $x_p^p + px_{p^2} + pf(x_1,x_2) = 1$, as desired.
\end{proof}
\item $ \ref{pthrootsmodp} \Rightarrow \ref{lev1surj}$
\begin{proof}
% Typically such arguments are valid only when $R$ is $p$-torsion free, but here we need only the injectivity of $w_1: R \rightarrow R$, which always holds.  
For any $r \in R$, we must find $r_1, \ldots, r_{p^{n}}$ such that $\sum_{i = 0}^{n} p^i r_{p^i}^{p^{n-i}} = r$.  We first find $r_1,s$ such that $r - r_1^{p^{n}}= ps$ by repeatedly applying \ref{pthrootsmodp}.  To find the remaining $r_{p^i}$, we apply the induction hypothesis to $s$.  
\end{proof}
\item $\ref{lev1surj} \Rightarrow \ref{pthrootsmodp}$; $\ref{lev1surj}' \Rightarrow \ref{lev1surj}$
\begin{proof}
We have already seen $ \ref{lev1surj}' \Rightarrow \ref{pthrootsmodp}$.  The two results follow because we have also shown $\ref{lev1surj} \Rightarrow \ref{lev1surj}'$ and $\ref{pthrootsmodp} \Rightarrow \ref{lev1surj}$.
\end{proof}
\item $\ref{lev1surj} + \ref{some power} \Rightarrow \ref{prmodp^2}$
\begin{proof}
By \ref{some power}, we can find $s_1,s_2 \in R$ for which $s_1^p = -p(1-s_2)$ and $s_2^N \in (p)$ for some $N > 0$.  We know that $\ref{lev1surj} \Rightarrow \ref{lev1surj}' \Rightarrow \ref{pthrootsmodp}$.
Given any $r \in R$, by \ref{pthrootsmodp} we can find
$s_3 \in R$ for which $s_3^p \equiv -r (1 + s_2 + \cdots + s_2^{N-1}) \mod{p}$. Since $s_2^N \equiv 0 \mod{p}$,
for $s = s_1 s_3$ we have $s^p = pr(1-s_2)(1+s_2 + \cdots + s_2^{N-1}) = pr(1-s_2^N) \equiv pr \mod{p^2}$.
\end{proof}
\item $\ref{lev1surj} + \ref{some power} \Rightarrow \ref{finlevsurj}$
\begin{proof}
We just saw that $\ref{lev1surj} + \ref{some power} \Rightarrow \ref{prmodp^2}$, and we also know that $\ref{prmodp^2} \Rightarrow  \ref{Vinimagefin2}$.  We will thus use \ref{Vinimagefin2} freely below.

We prove that $F: W_{p^n}(R) \rightarrow W_{p^{n-1}}(R)$ is surjective for $n \geq 1$ by induction on $n$. The base case $n = 1$ is exactly 
\ref{lev1surj}$'$.  Now assume the result for some fixed $n-1$, pick any $\underline{y} \in W_{p^n}(R)$, and consider the diagram 
$$\xymatrix{
 W_{p^{n+1}}(R) \ni \underline{r} \ar[dr]^{\text{res}} \ar[r]^F & {}\underline{y}' \in W_{p^n}(R) \ar[dr]^{\text{res}} & {}\underline{y}\in W_{p^n}(R) \ar[d]^{\text{res}}    \\
 & W_{p^n}(R) \ni \underline{s} \ar[r]^F & {}\underline{y}|_{W_{p^{n-1}}(R)}.
}$$
The term $\underline{s}$ exists by our inductive hypothesis and the term $\underline{r}$ exists because restriction maps are surjective.
If we had $\underline{y} = \underline{y}'$, we would be done.

Find $\underline{x}' \in W_{p^{n+1}}(R)$ with $F(\underline{x}') = V^n([1])$ using \ref{Vinimagefin2}.  Then find $\underline{x}'' \in W_{p^{n+1}}(R)$ with $\underline{y}-\underline{y}' = V^n(F^{n+1}(\underline{x}''))$ using \ref{lev1surj}.  Then a calculation shows 
$F\left(\underline{r} + \underline{x}'\underline{x}''\right)  = 
\underline{y},$ as desired.
\end{proof}
\item $\ref{pthrootsmodfinite}' \Rightarrow \ref{lev1surj} + \ref{some power}$; $\ref{pthrootsmodfinite} \Rightarrow \ref{prmodp^2}$
\begin{proof}
These are compositions of implications we have already proved.
\end{proof}

\item $\ref{finlevsurj}' \Rightarrow \ref{finlevsurj}; \ref{Teichinimagefin} \Rightarrow \ref{finlevsurj}; \ref{Teichinimagefin}' \Rightarrow \ref{Teichinimagefin}; \ref{pthrootsmodfinite} \Rightarrow \ref{Teichinimagefin}; \ref{pthrootsmodfinite}' \Rightarrow \ref{pthrootsmodfinite}$ 
\begin{proof}
We will prove that all six conditions appearing in the statement are equivalent.  We have already proven the following implications:
\[\xymatrix@R=.20in@C=.20in{
\text{\ref{finlevsurj}} \ar[r] \ar[d] & \text{\ref{Teichinimagefin}} \ar[r]\ar[d] & \text{\ref{pthrootsmodfinite}} \ar[d] \\
\text{\ref{finlevsurj}}' \ar[r] & \text{\ref{Teichinimagefin}}' \ar[r] & \text{\ref{pthrootsmodfinite}}'.
}\]
Thus, it suffices to prove that \ref{pthrootsmodfinite}$'$ implies \ref{finlevsurj}.  This follows because we have seen above that $\ref{pthrootsmodfinite}' \Rightarrow \ref{lev1surj} + \ref{some power} \Rightarrow \ref{finlevsurj}$.
\end{proof}
\end{itemize}

\section{Valuation rings}

Throughout this section, we assume that $R$ is a valuation ring with valuation $v$, in which $p$ is nonzero. 
In several cases, we also assume that $v$ is a real valuation.
This includes a number of the examples considered in Section \ref{eg section}.

\begin{Rmk} \label{In for valuation ring}
Suppose that $v(p)$ is $p$-divisible in the value group of $v$. 
Then for each nonnegative integer $n$, the ideal $I_n$ is principal,
generated by any $x \in R$ such that $v(x) = \left(\frac{1}{p} + \cdots + \frac{1}{p^n} \right) v(p)$.
If moreover 
$v$ is a real valuation and there exists $y \in R$ such that $v(y) = \frac{1}{p-1}v(p)$, then
$I_\infty$ is the principal ideal generated by $y$.
\end{Rmk}

\begin{Rmk}	
Condition \ref{sphericallycomp} holds whenever $v$ is a real valuation, $v(p)$ is $p$-divisible
(so the $I_n$ are as computed in Remark~\ref{In for valuation ring}), and 
$R$ is \emph{spherically complete} (i.e., any decreasing sequence of balls in $R$
has nonempty intersection). The spherically complete condition is in practice quite rare;
for instance, an infinite algebraic extension of $\Qp$ which is not discretely valued is never spherically complete.
As a result, \ref{sphericallycomp} is also rather rare, as then is \ref{surj}; see 
Example~\ref{OCpeg}.
\end{Rmk}

\begin{Rmk} \label{finlevsurj for valuation rings}
Condition \ref{finlevsurj} implies \ref{pthrootsmodp} (the Frobenius homomorphism on $R/pR$ is surjective)
and that there exists an element $x \in R$ with $0 < v(x) < v(p)$ (e.g., by \ref{pmodp^2new}). The converse is also true, as follows. By \ref{pthrootsmodp}, there exist $y,z \in R$ with $y^p \equiv x \mod{p}$, $z^p \equiv p/x \mod{pR}$.
Since $0 < v(x), v(p/x) < v(p)$, we have $v(y) = \frac{1}{p} v(x), v(z) = \frac{1}{p} (v(p) - v(x))$,
so $v(yz) = \frac{1}{p}$. Therefore, $u := (yz)^p/p$ is a unit in $R$. By \ref{pthrootsmodp} again, there exists $v \in R$ such that $v^p \equiv -u^{-1} \mod{pR}$.  Thus we have $puv^p \equiv -p \mod p^2R$.  Thus $(yzv)^p \equiv -p \mod{p^2R}$. This implies
\ref{pmodp^2new}, which together with \ref{pthrootsmodp} implies \ref{finlevsurj}.
As a byproduct of the argument, we note that \ref{finlevsurj} implies that $v(p)$ is $p$-divisible.
\end{Rmk}

\begin{Rmk}
For valuation rings, \ref{prmodp^2} implies \ref{finlevsurj}, and so the two conditions become equivalent.
To see this, note that if $R$ satisfies \ref{prmodp^2}, we can find $r_1$ such that $r_1^p = -p \mod{p^2R}$, and in particular, $r_1^p = -pu$ for some unit $u \in R$.  By \ref{prmodp^2} again, for any $x \in R$ we may find $r_2 \in R$ with $r_2^p = -pxu + p^2y$, with $y \in R$ and $u$ as above.   Since $p v(r_1) = v(-p) \leq v(-px) = p v(r_2)$, we have that $r_2/r_1$ is an element of $R$. We then compute $\left( \frac{r_2}{r_1}\right)^p = \frac{-pxu + p^2y}{-pu} = x-pu^{-1}y \equiv x \mod p$.  Hence \ref{pthrootsmodp} holds;
since \ref{prmodp^2} also implies \ref{some power}, we may deduce \ref{finlevsurj} as desired.
\end{Rmk}

\begin{Rmk}
If $R$ satisfies condition \ref{finlevsurj}, then it satisfies almost purity; see 
Section~\ref{section almost purity}. Thus if $S$ is the integral closure of $R$ in a finite
extension of $\mathrm{Frac}(R)$, then the maximal ideal of $S$ surjects onto the maximal ideal of $R$ 
under the trace map. In other words, $R$ is \emph{deeply ramified} in
the sense of Coates and Greenberg \cite{CG96}.
\end{Rmk}

\section{Examples} 
\label{eg section}

We now describe some simple examples realizing distinct subsets of the conditions considered above.

\begin{Eg} \label{pinveg}
Take $R$ to be any ring in which $p$ is invertible.  Then by Theorem~\ref{theorem on logical implications}, all of our conditions hold.  
\end{Eg}

\begin{Eg} \label{ZZeg}
Take $R = \ZZ$.  In this case, $I_i = (p)$ for all $i \geq 1$.
Thus \ref{some power} fails, and consequently, neither \ref{surj} nor \ref{finlevsurj} holds for $R = \ZZ$.
On the other hand, \ref{sphericallycomp} does hold for $R = \ZZ$. To see this,  we must show that any descending chain of balls $\cdots \supseteq B(r_{i-1},(p)) \supseteq B(r_i,(p)) \supseteq \cdots$ has nonempty intersection, which is clear.
\end{Eg}

\begin{Eg} \label{FpTeg}
Take $R = \FF_p[T]$.  In this case, \ref{prmodp^2} is satisfied trivially, because $pr = 0$ for all $r \in R$.  On the other hand, \ref{pthrootsmodp} is not satisfied.  
\end{Eg}

\begin{Eg} \label{OCpeg}
Take $R = \OCp$. Then \ref{pthroots} holds because $R$ is integrally closed in the algebraically
closed field $\Cp$); this implies that $R$ satisfies 
\ref{Teichinimage},
\ref{Vinimage},
\ref{Teichinimagefin},
\ref{Vinimageall}, \ref{Vinimagefin2},
\ref{Vinimagefin},
\ref{lev1surj},
\ref{pthroots},
\ref{pthrootsmodfinite},
\ref{prmodp^2},
\ref{pmodp^2new},
\ref{some power}, \ref{pthrootsmodp}.
On the other hand, \ref{sphericallycomp} does not hold by Lemma~\ref{OCp lemma}, so
$R$ does not satisfy \ref{surj}, \ref{Teichdense}, \ref{pinvertible}, \ref{sphericallycomp}.
\end{Eg}

\begin{Lem} \label{OCp lemma}
The ring $R = \OCp$ does not satisfy \ref{sphericallycomp}.
\end{Lem}
\begin{proof}
By Remark~\ref{In for valuation ring}, for $n$ a nonnegative integer, $I_n$ is the principal ideal generated by
$p^{\frac{1}{p} + \cdots + \frac{1}{p^n}}$, while $I_\infty$ is the principal ideal generated by $p^{\frac{1}{p-1}}$.
Each ball $B(r, I_\infty)$ contains an element which is algebraic over $\QQ$, since such elements are dense in $\Cp$ by Krasner's
lemma.  Furthermore, if two balls $B(r,I_{\infty})$ and $B(r',I_{\infty})$ intersect, they are in fact equal.  Therefore, there are only countably many such balls. On the other hand, one can construct uncountably many decreasing sequences $B(r_0, I_0) \supseteq B(r_1, I_1) \supseteq \cdots$ no two of which have the same intersection. For instance, 
take $x_0, x_1, \dots$ to be Teichm\"uller elements in $W(\mathbb{F}_p) \subseteq \OCp$, and put
\[
r_0 = x_0, r_1 = r_0 + x_1 p^{\frac{1}{p}}, r_2 = r_1  + x_2 p^{\frac{1}{p} + \frac{1}{p^2}}, \dots.
\]
Then any two of the resulting intersections $\cap_{i=0}^\infty B(r_i, I_i)$ are disjoint.
\end{proof}

\begin{Rmk}
It is possible to give a more constructive proof of Lemma~\ref{OCp lemma} using the explicit description of $\OCp$ given in \cite{Ked01}.
\end{Rmk}

\begin{Eg} \label{sphereeg}
Let $R$ denote the spherical completion of $\OCp$ constructed by Poonen in \cite{Poo93}.
We will show that $R$ satisfies \ref{surj}, and thus satisfies all of the labeled conditions except for 
\ref{pinvertible}.

We first recall the explicit construction of $R$.
Let $\ZZ_p\llbracket t^{\QQ} \rrbracket$ denote the ring of \emph{generalized power series}
over $\ZZ_p$; its elements are formal sums $\sum_{i \in \QQ, i \geq 0} c_i t^i$ with $c_i \in \ZZ_p$ such that
the set $\{i \in \QQ: c_i \neq 0\}$ is well-ordered. This ring is spherically complete for the $t$-adic valuation.
%:  given any sequence $i_0 \leq i_1 \leq \cdots$ of nonnegative rationals and any $x_0, x_1, \dots \in S\llbracket t^{\QQ} \rrbracket$ such that $x_j - x_{j+1}$ is divisible by $t^{i_j}$, we can find $x \in S\llbracket t^{\QQ} \rrbracket$ congruent to $x_j$ modulo $t^{i_j}$ for each $j \geq 0$. Explicitly, we take the coefficient of $t^k$ in $x$ to be equal to the coefficient of $t^k$ in $x_j$ if there exists $j \geq 0$ for which $k \leq i_j$ for some $j$, and $0$ otherwise.
Poonen's spherical completion of $\OCp$ is then the ring $\ZZ_p\llbracket t^{\QQ} \rrbracket/(t-p)$.
In particular, $R/(p) \cong \FF_p\llbracket t^{\QQ} \rrbracket/(t)$. 

{}From this description, it is clear that $R$ satisfies \ref{sphericallycomp} and \ref{pthrootsmodp}. Finally, since $R$
is a valuation ring and there exists $x \in R$ for which $0 < v(x) < v(p)$ (e.g., the image of $t^{1/p}$),
Remark~\ref{finlevsurj for valuation rings} implies that $R$ satisfies \ref{finlevsurj}. Putting this together,
we deduce that $R$ satisfies \ref{surj}.
\end{Eg}

\begin{Eg} \label{mup2eg}
Take $R = \ZZ[\mu_{p^2}]$, where $\mu_{p^2}$ is a primitive $(p^2)$-nd root of unity. 
Condition \ref{Vinimage} holds because the element $\underline{x} = \sum_{i=0}^{p-1} [\mu_{p^2}^i] \in W(R)$ satisfies
$F(\underline{x}) = V(1)$. (Since $R$ is $p$-torsion-free, this last equality can be checked at the ghost component level,
where it is apparent.)
On the other hand, \ref{pthrootsmodp} does not hold: the element $(1-\omega_{p^2})$ has $p$-adic valuation $\frac{1}{p(p-1)}$, but there is no element of $R$ which has $p$-adic valuation $\frac{1}{p^2(p-1)}$.
Similarly, \ref{prmodp^2} does not hold.
\end{Eg}

\begin{Eg} \label{mupinftyeg}
Take $R = \ZZ[\mu_{p^\infty}]$, i.e., the ring of integers in the maximal abelian extension of $\QQ$. We will see that \ref{finlevsurj} holds but \ref{Teichinimage} does not.
(The same analysis applies to $\ZZ_p[\mu_{p^{\infty}}]$ or its $p$-adic completion.)

Note that $R$ satisfies \ref{Vinimage} because $R$ contains the subring $\ZZ[\mu_{p^2}]$ which satisfies
\ref{Vinimage} by Example~\ref{mup2eg}. Thus to establish \ref{finlevsurj}, it is sufficient to check
condition \ref{pthrootsmodp}. 
%(One can also make this reduction by replacing $R$ by its completion, which is a valuation ring, then using Remark~\ref{finlevsurj for valuation rings}.)
For this, note that for any expression $a_1 \mu_{p^{i_1}} + \cdots + a_n
\mu_{p^{i_n}}$ with $a_1,\dots,a_n \in \ZZ$, we have $a_1 \mu_{p^{i_1}} + \cdots + a_n
\mu_{p^{i_n}} \equiv (a_1 \mu_{p^{i_1+1}} + \cdots + a_n
\mu_{p^{i_n+1}})^p \mod p$.

To establish that $R$ does not satisfy \ref{Teichinimage}, we will instead check that $R$ does not satisfy \ref{pthroots}.
We will do this assuming $p > 2$, by checking that
the congruence $x^p \equiv 1-p \mod{p I_\infty}$ has no solution. 
This breaks down for $p=2$ because $1-p = -1$ is the square of $i := \mu_4$;
in this case, one can show by a similar argument (left to the reader) that
the congruence $x^2 \equiv i+2\mu_8 \mod{2 I_\infty}$ 
has no solution.

Assume by way of contradiction that $p>2$ and there exists $x \in R$ for which $x^p - 1+p \in pI_\infty$.
Recall that by Remark~\ref{In for valuation ring},
$I_\infty$ is the principal ideal generated by $p^{1/(p-1)}$.
Choose an integer $n \geq 2$ for which $x \in \ZZ[\mu_{p^n}]$, and put $z = 1 - \mu_{p^n}$. 
By Lemma~\ref{power of mu difference} below, $z^{p^{n-1}-p^{n-2}} \equiv -p \pmod{z^{p^n}}$;
in particular, the $p$-adic valuation of $1-\mu_{p^n}$ is 
$\frac{1}{p^{n-1}(p-1)}$. There must thus exist $y \in \ZZ[\mu_{p^n}]$
such that $(1 + y z^{p^{n-1}-p^{n-2}})^p \equiv 1-p \pmod{z^{p^n}}$; subtracting $1$ from both sides and dividing by $p$,
we obtain the congruence $yz^{p^{n-1} - p^{n-2}}-y^p \equiv -1 \pmod{z^{p^{n-1}}}$.
Using the isomorphism $\ZZ[\mu_{p^n}]/(z^{p^{n-1}}) \cong \FF_p[T]/(T^{p^{n-1}})$ sending $z$ to $T$,
we obtain a solution $w$ of the congruence $w^p - w T^{p^{n-1}-p^{n-2}} \equiv 1 \pmod{T^{p^{n-1}}}$ in $\FF_p[T]$.
But no such solution exists: we cannot have $w \equiv 0 \pmod{T^{p^{n-2}}}$, so there is a largest index $i < p^{n-2}$
such that the coefficient of $T^i$ in $w$ is nonzero, which forces $T^{i+p^{n-1}-p^{n-2}}$ to appear with a nonzero
coefficient in $w^p - w T^{p^{n-1}-p^{n-2}}$.
\end{Eg}

\begin{Lem} \label{power of mu difference}
For $p$ an odd prime, $n\geq 2$ an integer, and $\mu_{p^n}$ a
primitive $p^n$-th root of unity,
in $\ZZ[\mu_{p^n}]$ we have $(1-\mu_{p^n})^{p^n-p^{n-1}} \equiv -p
\mod{(1-\mu_{p^n})^{p^n}}$.
\end{Lem}
\begin{proof}
Let $\Phi(T) = \sum_{i=0}^{p-1} T^{ip^{n-1}}$ denote the $p^n$-th
cyclotomic polynomial, so that we may identify $\ZZ[\mu_{p^n}]$ with
$\ZZ[T]/(\Phi(1-T))$  by identifying $\mu_{p^n}$ with $1-T$. As
\[
\Phi(1-T) \equiv \sum_{i = 0}^{p-1} (1-T^{p^{n-2}})^{pi} \mod{p^2},
\]
we have $T^{p^n-p^{n-1}} - \Phi(1-T) \equiv 0 \mod{p}$. Since the constant term of 
$\Phi(1-T)$ is $p$, it suffices to check that $\Phi(1-T) \equiv p \mod{(p^2, T^{p^{n-1}})}$;
this reduces to the case $n=2$. In this case, 
for $k=1,\dots,p-1$, the coefficient of $T^k$ in $\Phi(1-T)$ is
$\sum_{i=0}^{p-1} (-1)^k \binom{ip}{k}$. We may  write
$\binom{pi}{k} = \frac{ik}{p}
\binom{ip-1}{kp-1}$ and then invoke Lucas's criterion to
deduce that $\binom{ip-1}{kp-1} \cong \binom{p-1}{k-1}
\mod{p}$.
Therefore, $\sum_{i=0}^{p-1} (-1)^k \binom{ip}{k} \equiv (-1)^k
\binom{p}{k} \sum_{i=0}^{p-1} i \equiv 0 \mod{p^2}$,
as desired.
\end{proof}

\section{Almost purity}
\label{section almost purity}

We conclude with one motivation for studying condition \ref{finlevsurj}: 
it provides a natural context for the concept of
\emph{almost purity}, as introduced by Faltings and studied more recently 
by the second author and Liu in \cite{KL11} and by Scholze in \cite{Sch11}. 
More precisely, \ref{finlevsurj} amounts to an absolute version (not relying on a valuation subring)
of the condition for a ring to be \emph{integral perfectoid} in the sense of Scholze.

We begin by defining the adjective \emph{almost}.
See \cite{GR03} for more general setting.

\begin{Def}
Let $R$ be a $p$-torsion-free ring which is integrally closed in $R_p := R[p^{-1}]$
and which satisfies condition \ref{finlevsurj}.
A \emph{$p$-ideal} of $R$ is an ideal $I$ of $R$ 
such that $I^n \subseteq (p)$ for some positive integer $n$.
An $R$-module $M$ is \emph{almost zero} if $IM = 0$ for every $p$-ideal $I$ of $R$. 
%A morphism of $R$-modules is \emph{almost injective/surjective} if its kernel/cokernel is almost zero,
%and \emph{almost bijective} (or an \emph{almost isomorphism}) if it is both almost injective and almost surjective.
\end{Def}

\begin{Thm} \label{almost purity}
Let $R$ be a $p$-torsion-free ring which is integrally closed in $R_p := R[p^{-1}]$
and which satisfies condition \ref{finlevsurj}. Let $S_p$ be a finite \'etale $R_p$-algebra, let $S_0$ be the integral
closure of $R$ in $S_p$, and let $S$ be any $R$-subalgebra of $S_0$ such that $S/S_0$ is an almost zero $R$-module.
\begin{enumerate}
\item[(a)]
The ring $S$ also satisfies condition \ref{finlevsurj}.
\item[(b)]
For any $p$-ideal $I$ of $R$, there exist a finite free $R$-module $F$ and $R$-module homomorphisms
$S \to F \to S$ whose composition is multiplication by some $t \in R$ for which $I \subseteq (t)$.
\item[(c)]
The image of $S$ under the trace pairing map $S_p \to \operatorname{Hom}_{R_p}(S_p, R_p)$ is almost equal to 
the image of the natural map from $\operatorname{Hom}_R(S,R)$ to $\operatorname{Hom}_{R_p}(S_p, R_p)$.
\end{enumerate}
\end{Thm}
\begin{proof}
For each $t \in \mathbb{Q}$, choose integers $r,s \in \mathbb{Z}$ with $s > 0$ and $r/s = t$.
Since $R$ is integrally closed in $R_p$, the set
\[
R_t := \{x \in R[p^{-1}]: p^{-r} x^s \in R\}
\]
depends only on $t$. The function $v: R_p \to (-\infty, +\infty]$ given by
\[
v(x) := \sup\{t \in \mathbb{Q}: x \in R_t\}
\]
satisfies $v(x-y) \geq \min\{v(x), v(y)\}$,
$v(xy) \geq v(x) + v(y)$, and $v(x^2) = 2 v(x)$. 

Let $A$ be the separated completion of
$R_p$ under the norm $|\cdot| = e^{-v(\cdot)}$, and define the subring $\mathfrak{o}_A = \{x \in A: |x| \leq 1\}$
and the ideal $\mathfrak{m}_A = \{x \in A: |x| < 1\}$.
Let $\psi: R_p \to A$ be the natural homomorphism; then $\psi^{-1}(\mathfrak{o}_A)$ contains $R$ but may be larger. However, we do have $\psi^{-1}(\mathfrak{m}_A) \subset R$.

Since \ref{finlevsurj} implies \ref{pthrootsmodp} and \ref{prmodp^2}, we can choose $x_1, x_2 \in R$ with 
\[
x_1^p \equiv -p \mod{p^2R}, \qquad
x_2^p \equiv x_1 \mod{p R}.
\]
Then $\psi(x_1), \psi(x_2)$ are units in $A$, and for all $y \in A$,
\[
|\psi(x_1) y| = p^{-1/p} |y|, \qquad |\psi(x_2) y| = p^{-1/p^2} |y|.
\]

Given $\overline{y} \in \mathfrak{o}_A/(p)$, choose $y \in R[p^{-1}]$ so that $\psi(y)$ lifts $\overline{y}$.
Then $x_2^p y \in \psi^{-1}(\mathfrak{m}_A) \subset R$, so since $R$ satisfies \ref{finlevsurj},
we can find $z \in R$ with $x_2^p y \equiv z^p \mod{pR}$.
The element $\psi(z/x_2) \in \mathfrak{o}_A$ has the property that $\psi(z/x_2)^p \equiv \psi(y) \mod{(p/\psi(x_2)^p) \mathfrak{o}_A}$;
it follows that Frobenius is surjective on $\mathfrak{o}_A/(\psi(x_1)^{p-1})$.
This implies that Frobenius is also surjective on $\mathfrak{o}_A$: 
given $y \in \mathfrak{o}_A$, we can first find $z_0, y_1 \in \mathfrak{o}_A$ with
$y = z_0^p + \psi(x_1)^{p-1} y_1$, then find $z_1,y_2 \in \mathfrak{o}_A$
with $y_1 = z_1^p + \psi(x_1)^{p-1} y_2$, and then
$z := z_0 + \psi(x_2)^{p-1} z_1$ will have the property that $z^p \equiv y \mod{p \mathfrak{o}_A}$.
That is, $\mathfrak{o}_A$ also satisfies \ref{pthrootsmodp};
since \ref{pmodp^2new} is evident (using $x_1$), $\mathfrak{o}_A$ satisfies \ref{finlevsurj}.

Put $B = A \otimes_R S$, and extend $\psi$ by linearity to a homomorphism $\psi: S_p \to B$. By \cite[Theorem~3.6.12]{KL11}, 
there is a unique power-multiplicative norm on $B$ under which it is a finite
Banach $A$-module, and for this norm the subring $\mathfrak{o}_B = \{x \in B: |x| \leq 1\}$ also satisfies \ref{finlevsurj}.
As in \cite[Remark~2.3.14]{KL11}, for $\mathfrak{m}_B = \{x \in B: |x| < 1\}$, we have $\psi^{-1}(\mathfrak{m}_B) \subset S$.

Given $\overline{y} \in S/(p)$,
choose a lift $y \in S$ of $\overline{y}$. Since $B$ satisfies \ref{finlevsurj} and $\psi(B[p^{-1}])$ is dense in $S$,
we can find $z \in \psi^{-1}(\mathfrak{o}_B)$ for
which $u := z^p - y$ satisfies $|\psi(u)| \leq p^{-1}$. 
In particular, $u \in \psi^{-1}(\mathfrak{m}_B) \subset S$;
moreover, we may write $x_1^p = -p + p^2 w$ for some $w \in R$ and then write
\[
u = p (u/p) = (-x_1^p + p^2 w)(u/p) = -x_1 (x_1^{p-1} u/p) + p u w.
\]
The quantity $x_1^{p-1} u/p$ again belongs to $\psi^{-1}(\mathfrak{m}_B) \subset S$, so
$u \in (x_1,p) S$. Therefore Frobenius is surjective on $S/(x_1,p)$; by arguing as before (using the fact that
$x_2^p \equiv x_1 \mod{pR}$), we deduce that Frobenius is surjective on $S/(x_1^i,p)$ for $i=2,\dots,p$.
Therefore, $S$ satisfies \ref{pthrootsmodp};
since \ref{pmodp^2new} is again evident, $S$ satisfies \ref{finlevsurj}. This proves (a).
The proofs of (b) and (c) similarly reduce to the corresponding statements about $\mathfrak{o}_A$ and $\mathfrak{o}_B$,
for which see \cite[Theorem~5.5.9]{KL11} or \cite{Sch11}.
\end{proof}
\begin{Cor}
For $R$ and $S$ as in Theorem~\ref{almost purity}, $\Omega_{S/R} = 0$.
\end{Cor}
\begin{proof}
Since $S[p^{-1}]$ is finite \'etale over $R[p^{-1}]$, $\Omega_{S/R}$ is killed by $p^n$ for some nonnegative integer $n$.
If $n>0$, then for each $x \in S$ we may apply Theorem~\ref{almost purity} to write $x = y^p + pz$.
Then $dx = py^{p-1}\,dy + p\,dz$ is also killed by $p^{n-1}$. By induction, it follows that we may take $n=0$,
proving the claim.
\end{proof}
\begin{Rmk}
Note that the proof of Theorem~\ref{almost purity} involves the facts that
$\psi(R)$ and $\mathfrak{o}_A$ are almost isomorphic (using $\mathfrak{o}_A$ to define \emph{almost}), 
as are $\psi(S)$ and $\mathfrak{o}_B$.
Also, Theorem~\ref{almost purity} can be applied with $S_p = R_p$, to show that any $R$-subalgebra $R'$ of $R$
for which $R/R'$ is almost zero also satisfies \ref{finlevsurj}.
\end{Rmk}

\section*{Acknowledgements} 

The authors thank Laurent Berger, James Borger, Lars Hesselholt, Abhinav Kumar, 
Ruochuan Liu, Joe Rabinoff, and Liang Xiao for helpful discussions.

\bibliography{padicHodge3}
\bibliographystyle{plain} 

\end{document}